\documentclass[11pt,reqno]{amsart}

\usepackage{eucal,mathrsfs}

\usepackage[utf8]{inputenc}

\numberwithin{equation}{section}

\usepackage[a4paper,margin=1.2in]{geometry}
                		
\usepackage{graphicx}				
										
\usepackage{amssymb}

\usepackage[utf8]{inputenc}

\usepackage[polish, english]{babel}

\usepackage[T1]{fontenc}

\usepackage{mathtools}

\usepackage{tikz-cd}

\usepackage{amsthm}

\usepackage[toc]{appendix}

\usepackage{calligra}

\usepackage{hyperref}

\newtheorem{thm}{Theorem}[section]
\newtheorem{lem}[thm]{Lemma}

\theoremstyle{remark}

\newcommand{\cD}{{\mathcal D}}

\newcommand{\cO}{{\mathcal O}}

\renewcommand{\AA}{{\mathbb A}}

\newcommand{\CC}{{\mathbb C}}
\newcommand{\WW}{{\mathbb W}}
\newcommand{\ZZ}{{\mathbb Z}}

\newcommand{\oo}{{{\mathfrak o}_{K}}}

\newcommand{\im}{{\textnormal{im }}}

\newcommand{{\D}}{{\mathscr{D}_{X}}} %SHEAF OF DIFFERENTIAL OPERATORS%

\newcommand{\DR}{\mathrm{DR}} 

\newcommand{\Dd}{{\widehat{\mathcal{D}}}_{n}}
\newcommand{\Ddo}{{\widehat{\mathcal{D}}}_{n}^{0}}

\newcommand{\cDc}{{\widehat{\mathcal{D}}}} %COMPLETED RING OF DIFFERENTIAL OPERATORS%
\newcommand{\cDco}{{\widehat{\mathcal{D}}_{0}}} %COMPLETED INTEGRAL RING OF DIFFERENTIAL OPERATORS%

\title[Modules over completed Weyl algebras]{Modules of minimal dimension over completed Weyl algebras} %TYTUŁ%

\author{Feliks Rączka}

\address{Institute of Mathematics, University of Warsaw, ul.\ Banacha 2, 02-097, Warsaw, Poland}

  \address{Institute of Mathematics, Polish Academy of Sciences, ul.\ Śniadeckich 8,
    \newline\indent 00-656 Warsaw, Poland
  }

\email{fraczka@impan.pl}

\date{\today}

\begin{document}

 %Data??%

\begin{abstract}
We study the category of modules of minimal dimension over completed Weyl algebras in equal characteristic zero. In particular we prove finiteness of de Rham cohomology of such modules.
\end{abstract}

\maketitle

\section{Introduction}\label{Introduction}

This is the first in the series of papers devoted to the study of $\cD$-modules on rigid analytic varieties over the field $\CC(\!(z)\!)$ of formal Laurent series (and, more generally, over a discretely valued nonarchimedean field of equal characteristic zero). The goal of these papers is to prove that on a (quasi-compact and quasi-separated) smooth rigid variety de Rham cohomology groups of holonomic $\cD$-modules have finite dimensions. In this paper we deal with the analogous algebraic problem concerning modules of minimal dimension over completed Weyl algebras.

Let $K$ be a discretely valued nonarchimedean field and let $\oo$ be its valuation ring. We also fix a uniformizer $\varpi\in \oo.$ Let $\WW_{n}(\oo)$ denote the $n$-th Weyl algebra over $\oo$, i.e., the noncommutaitve $\oo$-algebra generated by $x_{1},\dots,x_{n},\partial_{1},\dots,\partial_{n}$ with relations $[x_{i},x_{j}]=[\partial_{i},\partial_{j}]=0$ and $[\partial_{i},x_{j}]=\delta_{ij}.$ We set
\begin{equation}\label{IntegralCompletedWeyl}
\Ddo=\varprojlim\frac{\WW_{n}(\oo)}{\varpi^{s+1}\WW_{n}(\oo)}
\end{equation}
and define the $n$-th \textit{completed Weyl algebra} over $K$ as
\begin{equation}\label{CompletedWeyl}
\Dd=\Ddo\otimes_{\oo}K
\end{equation}
Algebraic properties of completed Weyl algebras have been studied by many authors, for example by L.\ Narv\'aez Macarro in \cite{Narvaez}, and more recently by A.\ Pangalos in \cite{Pangalos}.

If $M$ is a left $\Dd$-module then we can consider the \textit{de Rham complex} of $M$
\begin{equation}\label{DRComplex}
\DR^{s}(M)=\bigoplus_{|I|=s}M.dx_{I}
\end{equation}
where $I=(1\leq i_{1}<\dots<i_{s}\leq n)$ is a multi-index, $dx_{I}=dx_{i_{1}}\wedge\dots\wedge dx_{i_{s}}$ and the derivative $\delta$ is given by
\begin{equation}\label{deRhamdif}
\delta(m.dx_{I})=\sum_{i=1}^{n}\partial_{i}m.dx_{i}\wedge dx_{I}
\end{equation}
We then define \textit{de Rham cohomology} of $M$ as the cohomology of this complex
\begin{equation}\label{DRCohomology}
H^{i}_{dR}(M):=H^{i}(\DR^{\bullet}(M))
\end{equation}

The ring $\Dd$ is known to be left and right noetherian and and it has been shown by Pangalos that it has homological dimension $n$. A left $\Dd$-module is said to be \textit{of minimal dimension} if it is finitely generated and $\textnormal{Ext}^{i}_{\Dd}(M,\Dd)=0$ for $i<n$. This is the algebraization of the geometric notion of holonomicity in the sense that if $R$ is a ring of differential operators of the affine algebra of a smooth algebraic $\CC$-variety $X$ then left $R$-modules of minimal dimension correspond to holonomic $\cD_{X}$-modules.

In this paper we will prove the following theorem.

\begin{thm}\label{MainThm0}
Let $K$ be a discretely valued nonarchimedean field of equal characteristic zero and let $M$ be a left $\Dd$-module of minimal dimension. Then $\dim_{K}H^{i}_{dR}(M)<\infty$ for all $i$.
\end{thm}

Let $k$ be the residue field of $\oo.$ Then 

\[
\overline{\cD}_{n}=\Ddo/\varpi\Ddo=\Ddo\otimes_{\oo}k
\]
is isomorphic to the $n$-th Weyl algebra over $k$ and if $\textnormal{char }k=0$ then $\overline{\cD}_{n}$ is the ring of algebraic differential operators of the affine $n$-space $\AA^{n}_{k}.$

The idea for the proof of Theorem \ref{MainThm0} is to study $\Ddo$-modules and compare their properties \textit{``on the generic and the special fiber''}, i.e., after tensoring with $K$ and $k$ respectively.
If $M$ is a finitely generated left $\Dd$-module then a \textit{lattice} in $M$ is a finitely generated $\Ddo$-submodule $L\subset M$ such that $L\otimes\oo K=M.$ We write $\overline{L}$ for the left $\overline{\cD}_{n}$-module $L\otimes\oo k$ and call it the \textit{reduction} of $L.$ Then a more refined version of Theorem \ref{MainThm0} is the following theorem.
\begin{thm}\label{MainThm1}
Let $K$ be a discretely valued nonarchimedean field of equal characteristic zero and let $M$ be a finitely generated left $\Dd$-module. Then the following conditions are equivalent:

\begin{enumerate}

\item[i)] $M$ is of minimal dimension.

\item[ii)] There exists a lattice $L\subset M$ such that $\overline{L}$ is a $\overline{\cD}_{n}$-module of minimal dimension.

\item[iii)] For any lattice $L\subset M$ the reduction $\overline{L}$ is a $\overline{\cD}_{n}$-module of minimal dimension.

\end{enumerate}
If those equivalent conditions are satisfied, then moreover

\begin{enumerate}

\item[A)] The semisimplification of $\overline{L}$ does not depend on $L.$

\item[B)] We have $\dim_{K}H_{dR}^{i}(M)<\infty$ for all $i$ and the equality $\chi_{dR}(M)=\chi_{dR}(\overline{L})$ holds.
\end{enumerate}
\end{thm}
Here 
\[
\chi_{dR}(M)=\sum(-1)^{i}\dim_{K}H_{dR}^{i}(M)
\]
is the \textit{Euler characteristic} for de Rham cohomology and $\chi_{dR}(\overline{L})$ is the Euler characteristic for de Rham cohomology of the holonomic $\cD_{\AA^{n}_{k}}$-module $\overline{L}$, which is known to be finite since $\textnormal{char }k=0.$ This is the main reason why the assumption that $K$ is of equal characteristic zero is crucial. In fact, once we know that a left $\Dd$-module of minimal dimension has a lattice with reduction of minimal dimension then it is fairly easy to ``lift'' finiteness of de Rham cohomology thanks to Lemma \ref{PerfectComplex} (which is very general and has nothing to do with $\Dd$-modules).

It is easy to give an example of the failure of Theorem \ref{MainThm0} when $K$ is of mixed characteristic with $\textnormal{char }k=p>0.$ The Tate algebra $K\langle x \rangle$ may be considered as a left $\widehat{\mathcal{D}}_{1}$-module of minimal dimension and the de Rham complex is
\[
\frac{d}{dx}:K\langle x \rangle\to K\langle x \rangle
\]
Since $|p|<1$, any power series of form 
\begin{equation}\label{nonintegrable}
\sum_{n\geq0}a_{n}p^{n}x^{p^{n}-1}
\end{equation}
where $|a_{n}|=1$ is an element of $K\langle x \rangle.$ These elements cannot be integrated with respect to $x$ in the sense that the formal power series 
\[
\int\sum_{n\geq0}a_{n}p^{n}x^{p^{n}-1}=\sum_{n\geq0}a_{n}x^{p^{n}}
\]
is not convergent for $|x|=1$ and thus not an element of $K\langle x \rangle.$ We conclude that $\dim_{K}H_{dR}^{1}(K\langle x \rangle)$ is infinite. Note that if the residue characteristic of $K$ is zero, then $|p|=1$ and power series of form (\ref{nonintegrable}) are not elements of $K\langle x\rangle.$ More generally, in the case of equal characteristic zero, if $f\in K[\![x]\!]$ and $\frac{df}{dx}\in K\langle x\rangle$ then $f\in K\langle x \rangle,$ i.e., every element of $K\langle x \rangle$ can be integrated with respect to $x$ and $H_{dR}^{1}(K\langle x \rangle)=0$. Therefore Theorem \ref{MainThm0} works in this simple case.

In section \ref{Preliminaries} we recall basic results about completed Weyl algebras and modules of minimal dimension. In Section \ref{Algebras over discrete valuation rings} we discuss some more sophisticated properties of modules over $\oo$-algebras. In Section \ref{Proof} we apply results of previous sections to give the proof of Theorem \ref{MainThm1}.

\subsection*{Acknowledgements}
This work was supported by the project KAPIBARA funded by the European Research Council (ERC) under the European Union's Horizon 2020 research and innovation programme (grant agreement No 802787). Part of this work was created during author's visit in Palaiseau. This visit was supported by the ``Inicjatywa Doskonałości-Uczelnia Badawcza'' programme funded by the University of Warsaw. The author thanks Piotr Achinger and Adrian Langer for numerous discussions and helpful comments. The author thanks Stefano Alo\'e, Javier Fres\'an, and Gabriel Ribeiro for their hospitality during his stay in Palaiseau.

\section{Preliminaries}\label{Preliminaries}

From now till the end of this paper $K$ is a fixed discretely valued nonarchimedean field of equal characteristic zero. We let $\oo\subset K$ be the valuation ring and $k$ the residue field. We fix a uniformizer $\varpi\in \oo$. Although we will not use it, we note that by the Cohen structure theorem a choice of a uniformizer gives an isomorphism $K=k(\!(\varpi)\!).$

\subsection{Modules of minimal dimension}\label{Dual1}\label{minimaldimension1}
In this subsection we recall basic properties of modules of minimal dimension and basic properties of holonomic $\cD_{\AA^{n}_{k}}$-modules.

Let $R$ be a (not necessarily commutative) ring. We recall that the projective dimension of an $R$-module $M$ (written $pd(M)$) is the infimum of lengths of its projective resolutions. Then we define the left global dimension of $R$ as
\[
\textnormal{l.gl.dim}(R)=\sup\{pd(M):\textit{$M$ is a left $R$-module}\}
\]
The right global dimension (written $\textnormal{r.gl.dim}(R)$) is defined in the same way. If $R$ is left and right noetherian then by \cite[Exercise 4.1.1]{Weibel} left and right global dimensions of $R$ are equal and we define the \textit{global dimension} of $R$ as
\[
\textnormal{gl.dim}(R)=\textnormal{l.gl.dim}(R)=\textnormal{r.gl.dim}(R)
\]
Now, let $R$ be a left and right noetherian ring of finite global dimension $\textnormal{gl.dim}(R)=n.$ Following \cite[1.2]{Meb2} we say that a finitely generated left (resp. right) $R$-module $M$  is \textit{of minimal dimension} if
\[
\inf\{i:\textnormal{Ext}^{i}_{R}(M,R)\neq0\}=n.
\]
For such module we set
\[
M^{*}=\textnormal{Ext}^{n}_{R}(M,R)
\]
The following Lemma is well-known. Since it is an important ingredient in the proof of Theorem \ref{MainThm1} we sketch the proof for completeness. 

\begin{lem}\label{star}
Let $M$ be a left (resp. right) $R$-module of minimal dimension. Then $M^{*}$ is a right (resp. left) $R$-module of minimal dimension and $M^{**}=M.$
\end{lem}

\begin{proof}
It is well known that if P is a finitely generated projective left (resp. right) module, then its dual $P^{\vee}=\textnormal{Hom}_{R}(P,R)$ is a finitely generated projective right (resp. left) module and the natural map $P\to P^{\vee\vee}$ is an isomorphism. Let $M$ be a left (resp. right) $R$-module of minimal dimension.

Since $R$ is noetherian we know that $M$ admits a finite projective resolution by finitely generated projective modules. Let $P_{\bullet}$ be such resolution and let $Q_{\bullet}=\textnormal{Hom}_{R}(P_{-\bullet},R)[n]$. First of all, we have $H_{i}(Q_{\bullet})=\textnormal{Ext}^{n-i}_{R}(M,R)$ and therefore $Q_{\bullet}$ is a projective resolution of $M^{*}.$ By reflexivity of finite projective modules we have $P_{\bullet}=\textnormal{Hom}_{R}(Q_{-\bullet},R)[n]$ and therefore

\begin{equation}
\textnormal{Ext}^{i}_{R}(M^{*},R)=H_{n-i}(P_{\bullet})=
\begin{cases}
0\textnormal{ if }i\neq n\\
M\textnormal{ if }i=n
\end{cases}
\end{equation}

This shows that $M^{*}$ is of minimal dimension and that $M^{**}=M.$
\end{proof}

If $R=\WW_{n}(k)$ for some field $k$ of characteristic zero then the category of left $R$-modules of minimal dimension coincides with the category of holonomic $\cD_{\AA^{n}_{k}}$-modules. This category is very well understood and for example the following properties of holonomic modules are well known.

De Rham complex and de Rham cohomology of a $\WW_{n}(k)$-module $M$ is defined by formulas (\ref{DRComplex}) and (\ref{DRCohomology}) of the introduction.

\begin{lem}\label{MinimalDimension}
 Let $M,M',M''$ be finitely generated left $\WW_{n}(k)$-modules. Then
\begin{enumerate}

\item[a)] If $0\to M'\to M\to M''\to 0$ is a short exact sequence then $M$ is of minimal dimension if and only if $M'$ and $M''$ are of minimal dimension.

\item[b)] $\textnormal{Ext}^{n}_{\WW_{n}(k)}(M,\WW_{n}(k))$ is a right $\WW_{n}(k)$-module of minimal dimension. 

\item[c)] If $M$ is of minimal dimension then it has finite length as a $\WW_{n}(k)$-module.

\item[d)] If $M$ is of minimal dimension then $\dim_{k}H^{i}_{dR}(M)<\infty$ for all $i.$

\end{enumerate}
The same holds with right and left replaced.
\end{lem}

\begin{proof}
Properties a), b) and c) are discussed in \cite[1.2]{Meb2}. Property a) is discussed after Definition 1.2.4 of op.\ cit.\ and b) is a consequence of Theorem 1.2.2 of op.\ cit.\ See also \cite[4.2.17]{Meb1} for the proof proof of b). Property $c)$ is \cite[Prop. 1.2.5]{Meb2}.  The last statement is a special case of the classical theorem of Bernstein which states that higher direct images in the derived category of $\cD$-modules preserve holonomicity. We refer to \cite[Ch. 1 Thm 6.1]{Bjork} for the proof of this theorem in the case of Weyl algebras and to \cite[Thm 3.2.3]{HTT} for the proof in full generality.
\end{proof}

\subsection{Completed Weyl Algebras}

In this subsection we discuss some basic properties of the completed Weyl Algebras, i.e, rings $\Ddo$ and $\Dd$ defined by formulas (\ref{IntegralCompletedWeyl}) and (\ref{CompletedWeyl}) of the introduction.

\begin{lem}\label{AlgProp2}
Both $\Ddo$ and $\Dd$ are left and right noetherian. Moreover $\textnormal{gl.dim}(\Dd)=n.$
\end{lem}

\begin{proof}
First of all the ring $\WW_{n}(\oo)$ is left and right noetherian. Indeed, the associated graded ring of the \textit{Bernstein filtration} $F_{n}\WW_{n}(\oo)=\bigoplus_{|\alpha|+|\beta|\leq n}a_{\alpha\beta}x^{\alpha}\partial^{\beta}$ is the polynomial ring in $2n$ variables over $\oo$. Since the valuation on $K$ is discrete $\oo$ is noetherian and therefore so is any polynomial ring over $\oo.$ We can apply \cite[Prop. D.1.4]{HTT} which states that if the associated graded ring is noetherian then so is the original ring.

It is well known in the commutative case that for a noetherian ring $R$ its $I$-adic completion is again noetherian. While this needs not be the case for noncommutative rings, it follows from \cite[Proposition 2.1.]{McC} that the Theorem remains true if $I$ is a two sided ideal generated by a single central element. Because $\WW(\oo)$ is left and right noetherian and $\varpi$ is central we conclude that $\Ddo$ is left and right noetherian. Then $\Dd$ is left and right noetherian because it is a localization of $\cDco$ at $\varpi.$ This proves the first part of the Lemma.

The ``moreover'' part follows from the PhD thesis of A.\ Pangalos \cite{Pangalos}. More precisely, Proposition 3.1.3. of op.\ cit.\ gives a bound $\textnormal{gl.dim}(\Dd)\geq n$  and Proposition 4.3.6. gives a bound $\textnormal{gl.dim}(\Dd)\leq n.$
\end{proof}

We will also need the following properties of $\Ddo$-lattices. Recall that $\overline{L}=L\otimes_{\oo}k.$

\begin{lem}\label{Lattices0}
Let $L$ be a finitely generated left $\Ddo$-module. Then:

\begin{enumerate}

\item[a)] $L$ is complete in the $\varpi$-adic topology.

\item[b)] If $\overline{L}=0$ then $L=0.$

\end{enumerate}

\end{lem}

\begin{proof}
Since $\varpi$ is central we can use the same reasoning as in the case of commutative noetherian rings. Since by \cite[p. 413]{Row} the Artin--Rees Lemma holds for finitely generated left $\Ddo$-modules we can proceed as in \cite[Ch.\ 10]{Atiyah} to check that if $I=(\varpi)$ and $L$ is a finitely generated left $\Ddo$-module, then
\[
\widehat{L}^{I}=\widehat{\Ddo}^{I}\otimes_{\Ddo}L=\Ddo\otimes_{\Ddo}L=L
\]
This proves the first statement of the lemma and the second statement follows from Nakayama's lemma for separated modules because $L=\widehat{L}^{I}$ is separated.
\end{proof}

\section{Algebras over discrete valuation rings}\label{Algebras over discrete valuation rings}

It is possible that Lemmas \ref{DVRsection1}, \ref{DVRsection2} and \ref{DVRsection3} below are known to the experts but we are not aware of any published proof in the form that we need. We will only use these results in case of the ring $\Dd$ to prove Theorem \ref{MainThm1} but since the proofs would not became easier nor shorter after restricting to this special case we present them in a more general setting. 

For the purpose of this section we assume that $B_{0}$ is a (not necessarily commutative) ring and $\pi\in B_{0}$ is a central element that is not a zero divisor. We set $B=B_{0}[\pi^{-1}]$ and $\overline{B}=B_{0}/\pi B_{0}.$ Because $\pi$ is not a zero divisor the natural map $B_{0}\to B$ is injective and we may write $B=\bigcup_{n\in\ZZ}\pi^{n}B_{0}.$ A model example of this situation is when $\pi$ is a uniformizer of some discrete valuation ring $\cO$ and $B_{0}$ is a flat $\cO$-algebra.

\subsection{Künneth type short exact sequences}\label{DVRsection1}

\begin{lem}\label{Kunneth}
Let $M$ be a right $B_{0}$-module that is $\pi$-torsion free and has a finite projective resolution by finitely generated modules. Then for each $i\geq0$ there are short exact sequences of left $\overline{B}$-modules
\[
0\to\overline{B}\otimes_{B_{0}}\textnormal{Ext}^{i}_{B_{0}}(M,B_{0})\to\textnormal{Ext}^{i}_{\overline{B}}(M\otimes_{B_{0}}\overline{B},\overline{B})\to\textnormal{Tor}_{1}^{B_{0}}(\overline{B},\textnormal{Ext}^{i+1}_{B_{0}}(M,B_{0}))\to0
\]
The same holds with left and right replaced and obvious modifications.
\end{lem}

\begin{proof}
Note that $\overline{B}$ has a projective resolution
\begin{equation}\label{pi-resolution}
0\to B_{0}\xrightarrow{\times\pi} B_{0}\to \overline{B}\to0
\end{equation}
Thus for any right $B_{0}$-module $M$ we have $\textnormal{Tor}_{i}^{B_{0}}(M,\overline{B})=0$ for $i\geq2$ and
\[
\textnormal{Tor}_{1}^{B_{0}}(M,\overline{B})=\{m\in M:m\pi=0\}
\]
In particular, if $M$ is $\pi$-torsion free and if
\[
P^{\bullet}=[0\to P^{-n}\to\dots\to P^{-1}\to P^{0}\to 0]
\]
is a projective resolution of $M$ by finitely generated modules then 
\[
\overline{P}^{\bullet}=P^{\bullet}\otimes_{B_{0}}\overline{B}
\]
is a projective resolution of $M\otimes_{B_{0}}\overline{B}$.  Set 
\begin{equation}\label{Qcomplex}
Q_{\bullet}=\textnormal{Hom}_{B_{0}}(P^{\bullet},B_{0})
\end{equation}
This is a complex of finitely generated projective left $B_{0}$-modules and we have 
\begin{equation}\label{ext1}
H_{i}(Q_{\bullet})=\textnormal{Ext}^{-i}_{B_{0}}(M,B_{0})
\end{equation}
On the other hand
\begin{equation}\label{ext2}
\textnormal{Ext}^{-i}_{\overline{B}}(M\otimes_{B_{0}}\overline{B},\overline{B})=H_{i}(\textnormal{Hom}(\overline{P}^{\bullet},\overline{B}))=H_{i}(\overline{B}\otimes_{B_{0}}Q_{\bullet})
\end{equation}
Here the first equality holds because $\overline{P}^{\bullet}$ is a projective resolution of $M\otimes_{B_{0}}\overline{B}$ and the second equality holds because for any finitely generated projective right $B_{0}$-module $P$ we have natural isomorphisms of left $\overline{B}$-modules
\[
\overline{B}\otimes_{B_{0}}\textnormal{Hom}_{B_{0}}(P,B_{0})=\textnormal{Hom}_{B_{0}}(P,\overline{B})=\textnormal{Hom}_{\overline{B}}(P\otimes_{B_{0}}\overline{B},\overline{B})
\]

Consider the following claim:
\textit{If $Q_{\bullet}$ is a bounded chain complex of finitely generated projective left $B_{0}$-modules then there exist exact sequences of left $\overline{B}$-modules}
\begin{equation}\label{claim}
0\to \overline{B}\otimes_{B_{0}}H_{j}(Q_{\bullet})\to H_{j}(\overline{B}\otimes_{B_{0}}Q_{\bullet})\to\textnormal{Tor}_{1}^{B_{0}}(\overline{B}, H_{j-1}(Q_{\bullet}))\to0
\end{equation}

Once we have proven the claim we are done with the proof because of equalities (\ref{ext1}) and (\ref{ext2}). The fastest way to show existence of exact sequences (\ref{claim}) is to use Künneth's spectral sequence \cite[Theorem 5.6.4]{Weibel} (see also \cite[Theorem 10.90]{Rot} for the formulation over noncommutative rings)
\begin{equation}\label{spectralsequence}
E^{2}_{i,j}=\textnormal{Tor}_{i}^{B_{0}}(\overline{B},H_{j}(Q_{\bullet}))\Rightarrow H_{i+j}(\overline{B}\otimes_{B_{0}}Q_{\bullet})
\end{equation}
and note that because of the resolution (\ref{pi-resolution}) we have $\textnormal{Tor}^{i}(\overline{B},-)=0$  for $i\neq0,1$ and the spectral sequence degenerates to short exact sequences
\begin{equation}\label{degeneracy}
0\to E^{2}_{0,j}\to H_{j}(\overline{B}\otimes_{B_{0}}Q_{\bullet})\to E^{2}_{1,j-1}\to0
\end{equation}

The problem with this approach is that in the literature existence of the spectral sequence (\ref{spectralsequence}) is usually formulated with $\overline{B}$ replaced by an arbitrary right $B_{0}$-module. Therefore formally one needs to check that maps in sequences (\ref{degeneracy}) are in fact $\overline{B}$-linear and not only additive (which is usually the case for tensor product of a left and a right module over a noncommutative ring). Alternatively, we can notice that if $d_{\bullet}$ is a differential in $Q_{\bullet}$ then as in the proof of \cite[Thm 3.6.1]{Weibel} we have the short exact sequence of complexes
\begin{equation}\label{361}
0\to\ker d_{\bullet}\otimes_{B_{0}}\overline{B}\to Q_{\bullet}\otimes_{B_{0}}\overline{B}\to \im d_{\bullet}\otimes_{B_{0}}\overline{B}\to0
\end{equation}
This is again a consequence of description of $\textnormal{Tor}^{i}(-,\overline{B}).$ Based on this observation we can copy the proof of \cite[Theorem 3.6.1]{Weibel} to prove our claim. Then $\overline{B}$-linearity is clear because the arrows in short exact sequences come from the long exact sequence associated to \ref{361}.
\end{proof}

\subsection{Reduction of lattices}\label{DVRsection2}

Lemma \ref{Lattices1} below is a simple generalization of a classical results that appear in many branches of mathematics. For example in  algebraic geometry a variant of Lemma \ref{Lattices1} for vector bundles with integrable connections is due to O.\ Gabber and may be found in a book of Katz \cite[Variant 2.5.2]{Katz}. More recently similar argument was used by A.\ Langer in \cite{langer2022moduli}. There is also a variant of Lemma \ref{Lattices1} in representation theory of finite groups over fields of positive characteristic (see \cite[Theorem 2.2.3]{Schneider})

Our definition of a lattice given in the introduction can be formulated in a more general setting as follows. A \textit{lattice} in a finitely generated $B$-module $M$ is a finitely generated $B_{0}$-submodule $L\subset M$ such that $B\otimes_{B_{0}}L=L[\pi^{-1}]=M.$ We set $\overline{L}=L/\pi L.$

Recall that a module $N$ is \textit{of finite length} if it has finite composition series $0=N_{0}\subset N_{1}\subset\dots\subset N_{r}=N$  where the factors $N_{i}/N_{i-1}$ are simple modules. The module $N^{\rm ss}=\bigoplus_{i=1}^{r}N_{i}/N_{i-1}$ does not depend on the choice of the composition series and is called the \textit{semisimplification} of $N.$

\begin{lem}\label{Lattices1}
Let $M$ be a finitely generated left $B$-module and let $L_{1},L_{2}\subset M$ be two lattices. If $\overline{L}_{1}$ has finite length then so does $\overline{L}_{2}$ and they have isomorphic semisimplifications.
\end{lem}

\begin{proof}
Since $B=\bigcup_{n\in\ZZ}\pi^{n}B_{0}$, there exist integers $n,m\in\ZZ$ with $\pi^{n}L_{2}\subset L_{1}\subset \pi^{m}L_{2}$. Because $\overline{\pi^{k}L}_{i}$ is isomorphic to $\overline{L}_{i}$ we may assume that
\begin{equation}\label{cont1}
L_{2}\subset L_{1}\subset \pi^{-n}L_{2}
\end{equation}
where $n\geq1.$ We prove the lemma by induction on $n$. We do the inductive step first. Assume that $n\geq2$ and that the statement is true for $n-1$. Then the result holds for $n$ because we have containments
\[
L_{2}\subset L_{1}\cap \pi^{-n+1}L_{2}\subset \pi^{-n+1}L_{2}
\]
and
\[
L_{1}\cap \pi^{-n+1}L_{2}\subset L_{1}\subset \pi^{-1}(L_{1}\cap \pi^{-n+1}L_{2})
\]

Therefore we only need to deal with the base for induction, i.e., with the case $n=1.$ We have
\begin{equation}\label{cont2}
L_{1}\subset \pi^{-1}L_{2}\subset \pi^{-1}L_{1}
\end{equation}

Taking reductions of (\ref{cont1}) (for $n=1$) and of (\ref{cont2}) gives exact sequences
\[
\overline{L}_{2}\xrightarrow{\varphi}\overline{L}_{1}\xrightarrow{\psi}\overline{L}_{2}
\]
and
\[
\overline{L}_{1}\xrightarrow{\psi}\overline{L}_{2}\xrightarrow{\varphi}\overline{L}_{1}
\]
where $\varphi$ (resp. $\psi$) is the map induced by the inclusion $L_{2}\subset L_{1}$ (resp. $L_{1}\subset\pi^{-1}L_{2}$). Therefore we have exact sequences
\begin{equation}\label{Ext1}
0\to\textnormal{im}\varphi\to\overline{L}_{1}\to\textnormal{im}\psi\to0
\end{equation}
and
\begin{equation}\label{Ext2}
0\to\textnormal{im}\psi\to\overline{L}_{2}\to\textnormal{im}\varphi\to0
\end{equation}

If $0\to N_{1}\to N\to N_{2}\to 0$ is a short exact sequence of modules then $N$ has finite length if and only if $N_{1}$ and $N_{2}$ have finite length. If this is a case then $N^{\rm ss}=N_{1}^{\rm ss}\oplus M_{2}^{\rm ss}.$ Therefore the result follows from existence of short exact sequences (\ref{Ext1}) and (\ref{Ext2}).
\end{proof}

\subsection{Euler characteristic of a perfect complex over a complete discrete valuation ring}\label{DVRsection3}

For the purpose the next lemma we will need much stronger assumptions. Let $B_{0}$ be a complete discrete valuation ring with the residue field $\ell$ and the quotient field $B$ (so $\ell=\overline{B}$ in the previous notation). We also fix a uniformizer $\pi\in B_{0}.$ Recall that a $B_{0}$-module $M$ is separated for the $\pi$-adic topology if $\bigcap_{n\geq0}\pi^{n}M=\{0\}$ and is complete if $M=\varprojlim_{n\geq0}M/\pi^{n+1}M.$ In particular complete modules are separated.

\begin{lem}\label{PerfectComplex}
Let $(C^{\bullet},d^{\bullet})$ be a complex of complete (for the $\pi$-adic topology), torsion-free $B_{0}$-modules and assume that all $\ell$-vector spaces $H^{i}(C^{\bullet}\otimes_{B_{0}}\ell)$ have finite dimensions. Then
\begin{enumerate}
\item[i)] All $B_{0}$-modules $H^{i}(C^{\bullet})$ are finitely generated and therefore also all $L$-vector spaces $H^{i}(C^{\bullet}\otimes_{B_{0}}B)$ have finite dimensions.

\item[ii)] If $C^{\bullet}$ is bounded then
\begin{equation}\label{Euler}
\sum(-1)^{i}\dim_{\ell}H^{i}(C^{\bullet}\otimes_{B_{0}}\ell)=\sum(-1)^{i}\dim_{L}H^{i}(C^{\bullet}\otimes_{B_{0}}B)
\end{equation}
i.e., the Euler characteristic of $C^{\bullet}$ on the special and the generic fibers are equal.
\end{enumerate}

\end{lem}
We need the following variant of Nakayama's Lemma.
\begin{lem}\label{Kule}
Let
\[
V\to W\to Q\to0
\]
be an exact sequence of $B_{0}$-modules. Assume that $V$ complete, $W$ is separated and $Q\otimes_{B_{0}}\ell$ is finitely generated. Then $Q$ is finitely generated.
\end{lem}
\begin{proof}
Since the tensor product is right exact we have a commutative diagram with exact rows.
\[
\begin{tikzcd}
V\arrow{r}{\varphi}\arrow{d}&W\arrow{r}{\psi}\arrow{d}&Q\arrow{r}\arrow{d}&0\\
\overline{V}\arrow{r}{\overline{\varphi}}&\overline{W}\arrow{r}{\overline{\psi}}&\overline{Q}\arrow{r}&0
\end{tikzcd}
\]
Pick generators $\overline{q}_{1},\dots,\overline{q}_{n}\in\overline{Q}=Q\otimes_{B_{0}}\ell$ and let $q_{1},\dots,q_{n}\in Q$ denote lifts of these elements to $Q$. Let $w_{1},\dots,w_{n}\in W$ satisfy $\psi(w_{i})=b_{i}.$ To prove the lemma it suffices to show that for any $x\in W$ there exist $r_{1},\dots,r_{n}\in B_{0}$ and $v\in V$ such that $w=\sum_{i=1}^{n}r_{i}w_{i}+\psi(v).$ Since $\overline{W}$ is generated mod $\im \varphi$ by $\overline{w}_{1},\dots,\overline{w}_{n}$, there exist $r_{1}^{0},\dots,r_{n}^{0}\in B_{0}$, $v_{0}\in V$ and $x_{1}\in W$ such that
\[
x=\sum_{i=1}^{n}r_{i}^{0}w_{i}+\varphi(v_{0})+\pi x_{1}
\]
We can repeat this process for $x_{1}$ to find inductively elements $r_{i}^{0},r_{i}^{1},r_{i}^{2},\dots,\in B_{0}$, $v_{0},v_{1},\dots\in V$ and $x_{1},x_{2},\dots\in W$ such that for every $m\geq1$
\[
x=\sum_{i=1}^{n}(\sum_{j=0}^{m}\pi^{j}r_{i}^{j})w_{i}+\sum_{j=0}^{m}\pi^{j}\varphi(v_{j})+\pi^{m+1}x_{m+1}.
\]
Since $B_{0}$ is complete there exist $r_{i}=\lim _{m\to\infty}\sum_{j=0}^{m}\pi^{j}r_{i}^{j}$. Since $V$ is complete there exists $v=\im_{m\to \infty}\sum_{j=0}^{m}\pi^{j}v_{j}$ and therefore $\varphi(v)=\lim_{m\to \infty}\sum_{j=0}^{m}\pi^{j}\varphi(v_{j}).$ Since $W$ is separated we have
\[
x-\sum_{i=1}^{n}r_{i}w_{i}-\varphi(x)\in\bigcap_{m\geq1}\pi^{m}W=\{0\}
\]
and therefore $x=\sum_{i=1}^{n}r_{i}w_{i}+\varphi(v)$ and we are done.
\end{proof}

\begin{proof}[Proof of Lemma \ref{PerfectComplex}]
Recall that a module over a discrete valuation ring is flat if and only if it is torsion-free. In particular images of $d^{\bullet}$ are also flat and we may invoke the Künneth formula \cite[Theorem 3.6.1]{Weibel}. We have exact sequences
\begin{equation}\label{Kunneth_dvr}
0\to H^{i}(C^{\bullet})\otimes_{B_{0}}\ell\to H^{i}(C^{\bullet}\otimes_{B_{0}}\ell)\to\textnormal{Tor}_{1}^{B_{0}}(H^{i+1}(C^{\bullet}),\ell)\to0
\end{equation}
To prove the first statement of the Lemma we consider the exact sequences
\begin{equation}\label{OntoHn}
C^{n-1}\xrightarrow{d^{n-1}}\ker d^{n}\to H^{n}(C^{\bullet})\to0
\end{equation}
By (\ref{Kunneth_dvr}) and assumptions of our theorem the dimensions $\dim_{\ell}H^{n}(C)\otimes_{B_{0}}\ell$ are finite. Moreover by assumption $C^{n}$ are all complete and thus $\ker d^{n}$ are separated modules as they are submodules of complete (and thus separated) modules. Therefore we may apply Lemma \ref{Kule} to sequences (\ref{OntoHn}) to conclude the first part of the Lemma.

For the second part, recall that it follows from the classification of finitely generated modules over discrete valuation rings that if $M$ is such module then
\begin{equation}\label{finiteDVR}
\dim_{\ell}M\otimes_{B_{0}}\ell-\dim_{B}M\otimes_{B_{0}}B=\dim_{k}\textnormal{Tor}^{B_{0}}_{1}(M,\ell)
\end{equation}
Since $-\otimes_{B_{0}}B$ is the same as localization at $\pi$ it is an exact functor. The first part of the lemma together with (\ref{Kunneth_dvr}) and (\ref{finiteDVR}) imply formula (\ref{Euler}). Indeed, we have
\begin{equation*}
\begin{split}
\sum(-1)^{i}\dim_{\ell}H^{i}(C^{\bullet}\otimes_{B_{0}}\ell)&=\sum(-1)^{i}\dim_{\ell}H^{i}(C^{\bullet})\otimes_{B_{0}}\ell+\sum(-1)^{i}\dim_{\ell}\textnormal{Tor}_{1}^{B_{0}}(H^{i+1}(C^{\bullet}),\ell)\\
&=\sum(-1)^{i}(\dim_{\ell}H^{i}(C^{\bullet})\otimes_{B_{0}}\ell-\dim_{\ell}\textnormal{Tor}_{1}^{B_{0}}(H^{i}(C^{\bullet}),\ell))\\
&=\sum(-1)^{i}\dim_{B}H^{i}(C^{\bullet})\otimes_{B_{0}}B\\
&=\sum(-1)^{i}\dim_{B}H^{i}(C^{\bullet}\otimes_{B_{0}}B).
\end{split} 
\end{equation*}
This finishes the proof.
\end{proof}

\section{Proof of Theorem \ref{MainThm1}}\label{Proof}

We now use lemmas proven in the previous section and some well-known properties of holonomic $\cD$-modules on affine spaces to prove Theorem \ref{MainThm1}.

\begin{proof}[Proof of Theorem \ref{MainThm1}]
First we prove that $i)\implies ii)\implies iii)\implies i)$. We show condition $A)$ as a part of the second implication. Then we show $B).$

$i)\implies ii).$ This is the most tricky part of the proof. It suffices to prove that for any \textit{right} $\Dd$-module $N$ of minimal dimension its dual $N^{*}=\textnormal{Ext}^{n}_{\Dd}(N,\Dd)$ has a lattice with reduction of minimal dimension. Indeed, by Lemma \ref{star} we then may take $N=M^{*}$ which is of minimal dimension and satisfies $N^{*}=M^{**}=M.$

Let $V\subset N$ be some lattice (\textit{a priori} with reduction that is possibly not of minimal dimension). By Lemma \ref{Kunneth} applied to $B_{0}=\Ddo$ and $\pi=\varpi$ we have an inclusion
\[
0\to\overline{\cD}\otimes_{\Ddo}\textnormal{Ext}_{\Ddo}^{d}(V,\Ddo)\to\textnormal{Ext}^{n}_{\overline{\cD}_{n}}(\overline{V},\overline{\cD}_{n})
\]
The key observation is that the module on the left is of minimal dimension. Indeed, since $\overline{\cD}_{n}=\WW_{n}(k)$ and $\overline{V}$ is finitely generated, by part $b)$ of Lemma \ref{MinimalDimension} we know that the module on the right hand side is of minimal dimension. Therefore so is the module on the left hand side by part $a)$ of the same lemma. Now set
\[
T=\{m\in \textnormal{Ext}_{\Ddo}^{n}(V,\Ddo): \varpi^{k}m=0\textnormal{ for some k}\}.
\]
This is a left $\Ddo$-module because $\varpi$ is central in $\Ddo.$  We define $L$ as the quotient of $ \textnormal{Ext}^{n}_{\cDc_{0}}(V,\cDc_{0})$ by $T$, so that it fits into a short exact sequence.
\begin{equation}\label{ses1}
0\to T\to \textnormal{Ext}^{n}_{\cDc_{0}}(V,\cDc_{0})\to L\to0.
\end{equation}
We will show that $L$ is the desired lattice, i.e., that

\begin{enumerate}

\item[(a)] $K\otimes_{\oo}L=N^{*}$ and the natural map $L\to N^{*}$ is injective.

\item[(b)] $L$ is a finitely generated $\Ddo$-module.

\item [(c)] $\overline{\cD}_{n}\otimes_{\Ddo}L$ has minimal dimension.

\end{enumerate}

To show (a) we note that $K\otimes_{\oo}-$ coincides with the localization at $\varpi.$ In particular it is an exact functor (and therefore it commutes with Ext) and (by construction) $K\otimes_{\oo}T=0.$ Tensoring (\ref{ses1}) with $K$ we get that $K\otimes_{\oo}L=N^{*}.$ The natural map $L\to M$ is injective by construction because its kernel consists precisely of $\varpi$-torsion of $L$. Recall that by Lemma \ref{AlgProp2} we know that $\Ddo$ is left and right noetherian. From noetherianity we conclude that because $V$ was finitely generated so is $\textnormal{Ext}_{\Ddo}^{n}(V,\Ddo)$. Then $L$ is also finitely generated because by (\ref{ses1}) it is a quotient of a finitely generated module and (b) follows. Since tensoring is right exact we have an exact sequence of left $\overline{\cD}_{n}$-modules
\[
\overline{\cD}_{n}\otimes_{\Ddo}\textnormal{Ext}^{n}_{\Ddo}(V,\Ddo)\to \overline{\cD}_{n}\otimes_{\Ddo}L\to0,
\]
we conclude (c) from part $a)$ of Lemma \ref{MinimalDimension} because the right hand side is a quotient of a $\overline{\cD}_{n}$-module which we have already observed to be of minimal dimension. Therefore the implication is proven.

$ii)\implies iii).$ This is a consequence of Lemma \ref{Lattices1}. Indeed, it is known \cite[Ch. 1 Prop. 5.3]{Bjork} that finitely generated $\overline{\cD}_{n}$-modules of minimal dimension are of finite length. It follows by induction from part $a)$ of Lemma \ref{MinimalDimension} that a semisimplification of a $\overline{\cD}_{n}$-module of finite length has minimal dimension if and only if the module itself has minimal dimension. Therefore we may use Lemma \ref{Lattices1} to get the desired implication. We also get $A)$ as a byproduct.

$iii)\implies i).$ Let $L\subset M$ be a lattice such that $\overline{L}$ has minimal dimension. By the very definition we have $\textnormal{Ext}_{\overline{\cD}}^{i}(\overline{L},\overline{\cD})=0$ for $0\leq i\leq n-1$.Then short exact sequences of Lemma \ref{Kunneth} for $B_{0}=\Ddo$ give
\[
0\to\textnormal{Ext}_{\Ddo}^{i}(L,\Ddo)\otimes_{\Ddo}\overline{\cD}_{n}\to\textnormal{Ext}^{i}_{\overline{\cD}_{n}}(\overline{L},\overline{\cD}_{n})=0
\]
i.e., $\textnormal{Ext}_{\cDc_{0}}^{i}(L,\cDc_{0})\otimes_{\cDc_{0}}\overline{\cD}=0$ for $i<n.$ By noetherianity of $\Ddo$ (Lemma \ref{AlgProp2}) we know that the right $\Ddo$-modules $\textnormal{Ext}_{\Ddo}^{i}(L,\Ddo)$ are finitely generated and therefore must be zero by Nakayama's lemma part of Lemma \ref{Lattices0}. As we have already explained while proving that $i)\implies ii)$, we always have isomorphisms $\textnormal{Ext}_{\cDc}^{i}(M,\cDc)=\textnormal{Ext}_{\cDc_{0}}^{i}(L,\cDc_{0})\otimes_{\oo}K$. We conclude that $\textnormal{Ext}_{\cDc}^{i}(M,\cDc)$ must vanish for $i<n,$ i.e., $M$ has minimal dimension. This closes the circle of implications.

To prove $B)$ we use Lemma \ref{PerfectComplex}. Assume that equivalent conditions of Theorem \ref{MainThm2} hold for $M$ and let $L\subset M$ be a lattice which has a reduction of minimal dimension. Consider the complex
\[
\DR^{\bullet}(L)=L\to\bigoplus_{i=1}^{d}Ldx_{i}\to\bigoplus_{i<j}Ldx_{i}\wedge dx_{j}\to\dots 
\]
with differentials as in (\ref{deRhamdif}). This is a bounded complex of complete (by Lemma \ref{Lattices0}) and torsion-free (since lattices are $\varpi$-torsion free) $\oo$-modules. Note that by construction we have
\[
\DR^{\bullet}(L)\otimes_{\oo}K=\DR^{\bullet}_{\Dd}(M)
\]
and
\[
\DR^{\bullet}(L)\otimes_{\oo}k=\DR^{\bullet}_{\overline{\cD}_{n}}(\overline{L})
\]
The latter has finitely-dimensional cohomology over $k$ by part $d)$ of Lemma \ref{MinimalDimension}. We may now apply Lemma \ref{PerfectComplex} and conclude that $\dim_{K}H_{dR}^{i}(M)<\infty$ for all $i$ and moreover $\chi_{dR}(M)=\chi_{dR}(\overline{L}).$
\end{proof}

%BIBLIOGRAFIA%

\bibliographystyle{plain}
\bibliography{Bibliography}

\end{document}